\documentclass[12pt]{amsart}
\usepackage{amsmath, amssymb}

\theoremstyle{plain}
\newtheorem{theorem}{Theorem}[section]
\newtheorem{proposition}[theorem]{Proposition}
\newtheorem{corollary}[theorem]{Corollary}
\newtheorem{lemma}[theorem]{Lemma}

\newtheorem{conjecture}[theorem]{Conjecture}

\newtheorem*{problem}{Problem}

\numberwithin{equation}{section}

\theoremstyle{definition}

\theoremstyle{remark}

\newtheorem*{remarks}{Remarks}

\DeclareMathOperator{\Var}{Var}
\DeclareMathOperator{\Span}{span}
\DeclareMathOperator{\conv}{conv}
\DeclareMathOperator{\ext}{ext}

\def \< {\langle}
\def \> {\rangle}

\def \R {\mathbb{R}}

\def \E {\mathbb{E}}
\def \P {\mathbb{P}}
\def \one {{\bf 1}}
\def \NN {\mathcal{N}}
\def \EE {\mathcal{E}}
\def \a {\alpha}
\def \b {\beta}

\def \e {\varepsilon}
\def \d {\delta}

\begin{document}

\title[Approximating covariance matrices]
  {How close is the sample covariance matrix to the actual covariance matrix?}
  
\author{Roman Vershynin}


\address{Department of Mathematics, University of Michigan, 530 Church Street, Ann Arbor, MI 48109, U.S.A.}
\email{romanv@umich.edu}

\thanks{Partially supported by NSF grant FRG DMS 0918623}

\begin{abstract}
Given a probability distribution in $\R^n$ with general (non-white) covariance, 
a classical estimator of the covariance matrix 
is the sample covariance matrix obtained from a sample of $N$ independent points. 
What is the optimal sample size $N = N(n)$ that guarantees estimation with a fixed accuracy
in the operator norm?
Suppose the distribution is supported in a centered Euclidean 
ball of radius $O(\sqrt{n})$. We conjecture that the optimal sample size is 
$N = O(n)$ for all distributions with finite fourth moment, and we prove this up to an iterated logarithmic factor.
This problem is motivated by the optimal theorem of M.~Rudelson \cite{R} which states that
$N = O(n \log n)$ for distributions with finite second moment, and 
a recent result of R.~Adamczak et al. \cite{ALPT} which guarantees that $N = O(n)$ for 
sub-exponential distributions.
\end{abstract}

\maketitle

\section{Introduction}

\subsection{Approximation problem for covariance marices}
Estimation of covariance matrices of high dimensional distributions is a basic problem 
in multivariate statistics. It arises in diverse applications such as 
signal processing \cite{KV}, 
genomics \cite{SS}, 
financial mathematics \cite{LW},
pattern recognition \cite{DKPN},
geometric functional analysis \cite{R} and computational geometry \cite{ALPT}.
The classical and simplest estimator of a covariance matrix is the sample covariance matrix. 
Unfortunately, the spectral theory of sample covariance matrices has not
been well developed except for product distributions (or affine transformations thereof) 
where one can rely on random matrix theory for matrices with independent entries. 
This paper addresses the following basic question: 
how well does the sample covariance matrix approximate the actual covariance matrix
in the operator norm?

We consider a mean zero random vector $X$ in a high dimensional space $\R^n$
and $N$ independent copies $X_1,\ldots,X_N$ of $X$.
We would like to approximate the covariance matrix of $X$
$$
\Sigma = \E X \otimes X = \E XX^T
$$
by the sample covariance matrix
$$
\Sigma_N = \frac{1}{N} \sum_{i=1}^N X_i \otimes X_i.
$$

\begin{problem}
  Determine the minimal sample size $N = N(n,\e)$ that guarantees with high probability 
  (say, $0.99$) that the sample covariance matrix $\Sigma_N$ approximates the actual 
  covariance matrix $\Sigma$ with accuracy $\e$ in the operator norm $\ell_2 \to \ell_2$, i.e. so that 
  \begin{equation}							\label{Sigma Sigman}
  \|\Sigma - \Sigma_N\| \le \e.
  \end{equation}
\end{problem}

The use of the operator norm in this problem allows one a good grasp of the spectrum of $\Sigma$, as 
each eigenvalue of $\Sigma$ would lie within $\e$ from the corresponding eigenvalue of $\Sigma_N$.

It is common for today's applications to operate with increasingly large number of parameters $n$,
and to require that sample sizes $N$ be moderate compared with $n$.
As we impose no a priori structure on the covariance matrix, 
we must have $N \ge n$ for dimension reasons.
Note that for some structured covariance matrices, such as sparse or having an off diagonal decay,
one can sometimes achieve $N$ smaller than $n$ and even comparable to $\log n$, 
by transforming the sample covariance matrix in order to adhere to the same 
structure (e.g. by shrinkage of eigenvalues or thresholding of entries).
We will not consider structured covariance matrices in this paper; 
see e.g. \cite{RLZ} and \cite{LV}.

\subsection{Two examples}
The most extensively studied model in random matrix theory is where $X$ is a random vector 
with independent coordinates.
However, independence of coordinates can not be justified in some important applications,
and in this paper we shall consider general random vectors.
Let us illustrate this point with two well studied examples.

Consider some non-random vectors $x_1, \ldots, x_M$ in $\R^n$ which satisfy Parseval's identity 
(up to normalization):
\begin{equation}							\label{parseval}
\frac{1}{M} \sum_{j=1}^M \< x_j, x \> ^2 = \|x\|_2^2
	\quad \text{for all } x \in \R^n.
\end{equation}
Such generalizations of orthogonal bases $(x_j)$ are called tight frames.
They arise in convex geometry via John's theorem on contact points of convex bodies
\cite{Ball} and 
in signal processing as a convenient mean to introduce redundancy into signal representations
\cite{KC}. 
From a probabilistic point of view, we can regard the normalized sum in \eqref{parseval} as 
the expected value of a certain random variable. 
Indeed, Parseval's identity \eqref{parseval} amounts to $\frac{1}{M} \sum_{j=1}^M x_j \otimes x_j = I$. 
Once we introduce a random vector $X$ uniformly distributed in the set of $M$ points $\{x_1, \ldots, x_M\}$, 
Parseval's identity will read as $\E X \otimes X = I$. 
In other words, the covariance matrix of $X$ is identity, $\Sigma=I$.
Note that there is no reason to assume that the coordinates of $X$ are independent.

Suppose further that the covariance matrix of $X$ can be approximated 
by the sample covariance matrix $\Sigma_N$ for 
some moderate sample size $N = N(n,\e)$. Such an approximation $\|\Sigma_N - I\| \le \e$ means simply that 
a random subset of $N$ vectors $\{ x_{j_1}, \ldots, x_{j_N} \}$ 
taken from the tight frame $\{x_1,\ldots,x_M\}$ independently and 
with replacement is still an approximate tight frame:
$$
(1-\e) \|x\|_2^2 \le \frac{1}{N} \sum_{i=1}^N \< x_{j_i}, x \> ^2 \le (1+\e) \|x\|_2^2
	\quad \text{for all } x \in \R^n.
$$
In other words, a small random subset of a tight frame is still an approximate tight frame;
the size of this subset $N$ does not even depend on the frame size $M$.
For applications of this type of results in communications see \cite{V frames}.

Another extensively studied class of examples is the uniform distribution on a convex body $K$ in $\R^n$.
A number of algorithms in computational convex geometry (for volume computing and optimization)
rely on covariance estimation in order to put $K$ in the isotropic position, see \cite{KLS, KR}.
Note that in this class of examples, the random vector uniformly distributed in $K$ typically does not 
have independent coordinates.

\subsection{Sub-gaussian and sub-exponential distributions}
Known results on the approximation problem differ depending on the moment assumptions on the
distribution. The simplest case is when $X$ is a {\em sub-gaussian random vector} in $\R^n$, thus satisfying 
for some $L$ that
\begin{equation}							\label{sub-gaussian}
\P (|\< X, x\> | > t) \le 2 e^{-t^2/L^2}
	\quad \text{for $t > 0$ and $x \in S^{n-1}$}.
\end{equation}
Examples of sub-gaussian distributions with $L = O(1)$ include the standard Gaussian random 
distribution in $\R^n$, the uniform distribution on the cube $[-1,1]^n$, but not the uniform distribution 
on the unit octahedron $\{ x \in \R^n :\; |x_1| + \cdots + |x_n| \le 1\}$.
For sub-gaussian distributions in $\R^n$, the optimal sample size in the approximation problem 
\eqref{Sigma Sigman} is linear in the dimension, thus $N = O_{L,\e}(n)$. 
This known fact follows from a large deviation inequality and an $\e$-net argument,
see Proposition \ref{subgaussian prop} below.

Significant difficulties arise when one tries to extend this result to the larger class 
of {\em sub-exponential random vectors} $X$, 
which only satisfy \eqref{sub-gaussian} with $t^2/L^2$ replaced by $t/L$. This class is important because,
as follows from Brunn-Minkowski inequality,
the uniform distribution on every convex body $K$ is sub-exponential 
provided that the covariance matrix is identity (see \cite[Section~2.2.(b$_3$)]{GM Euclidean}).
For the uniform distributions on convex bodies, a result of J.~Bourgain \cite{Bo} guaranteed
approximation of covariance matrices with sample size slightly larger than linear 
in the dimension, $N = O_\e(n \log^3 n)$.
Around the same time, a slightly better bound $N = O_\e(n \log^2 n)$ was proved by M.~Rudelson \cite{R}. 
It was subsequently improved to $N = O_\e(n \log n)$ 
for convex bodies symmetric with respect to the coordinate hyperplanes
by A.~Giannopoulos et al. \cite{GHT}, and for general convex bodies by G.~Paouris \cite{Pa}.
Finally, an optimal estimate $N = O_\e(n)$ was obtained by G.~Aubrun \cite{Au} 
for convex bodies with the symmetry assumption as above, and for general convex bodies 
by R.~Adamczak et al. \cite{ALPT}.
The result in \cite{ALPT} is actually valid for all sub-exponential distributions supported in a ball of 
radius $O(\sqrt{n})$. Thus, if $X$ is a random vector in $\R^n$ that satisfies for some $K,L$ that
\begin{equation}							\label{sub-exponential}
\|X\|_2 \le K \sqrt{n} \text{ a.s.,} \quad 
\P (|\< X, x\> | > t) \le 2 e^{-t/L}
	\quad \text{for $t > 0$ and $x \in S^{n-1}$}
\end{equation}
then the optimal sample size is $N = O_{K,L,\e}(n)$.

The boundedness assumption $\|X\|_2 = O(\sqrt{n})$ is usually non-restrictive, 
since many natural distributions satisfy this bound with overwhelming probability. 
For example, the standard Gaussian random 
vector in $\R^n$ satisfies this with probability at least $1-e^{-n}$. 
It follows by union bound that for any sample size $N \ll e^n$, 
all independent vectors in the sample $X_1,\ldots,X_N$ satisfy this inequality simultaneously
with overwhelming probability. Therefore, by truncation one may assume without loss of generality 
that $\|X\|_2 = O(\sqrt{n})$. A similar reasoning is valid for uniform distributions on convex bodies.
In this case one can use the concentration result of G.~Paouris \cite{Pa} 
which implies that $\|X\|_2 = O(\sqrt{n})$ with 
probability at least $1 - e^{-\sqrt{n}}$.

\subsection{Distributions with finite moments}
Unfortunately, the class of sub-exponential distributions is too restrictive 
for many natural applications.
For example, discrete distributions in $\R^n$ supported on less than $e^{O(\sqrt{n})}$ points
are usually not sub-exponential. 
Indeed, suppose a random vector $X$ takes values in some set of $M$ vectors of Euclidean length $\sqrt{n}$. 
Then the unit vector $x$ pointing to the most likely value of $X$ witnesses that $\P (|\< X, x\> | = \sqrt{n}) \ge 1/M$.
It follows that in order for the random vector $X$ to be sub-exponential with $L = O(1)$, 
it must be supported on a set of size $M \ge e^{c\sqrt{n}}$. 
However, in applications such as \eqref{parseval} it is desirable to have a result valid for distributions 
on sets of moderate sizes $M$, e.g. polynomial or even linear in dimension $n$.
This may also be desirable in modern statistical applications, which typically operate with large 
number of parameters $n$ that may not be exponentially smaller than the population size $M$.

So far, there has been only one approximation result with very weak assumptions on the distribution.
M.~Rudelson \cite{R} showed that if a random vector $X$ in $\R^n$ satisfies 
\begin{equation}							\label{second moment}
\|X\|_2 \le K \sqrt{n} \text{ a.s.,} \quad 
\E \< X, x\> ^2 \le L^2 
	\quad \text{for } x \in S^{n-1}
\end{equation}
then the minimal sample size that guarantees approximation \eqref{Sigma Sigman} 
is $N = O_{K,L,\e}(n \log n)$. The second moment assumption in \eqref{second moment}
is very weak; it is equivalent to the boundedness of the covariance matrix, $\|\Sigma\| \le L$. 
The logarithmic oversampling factor is necessary in this extremely general result, 
as can be seen from the example of the uniform 
distribution on the set of $n$ vectors of Euclidean length $\sqrt{n}$. 
The coupon collector's problem calls for
the size $N \gtrsim n \log n$ in order for the sample $\{X_1,\ldots,X_N\}$ to contain all these vectors,
which is obviously required for a nontrivial covariance approximation.

\medskip
There is clearly a big gap between the sub-exponential assumption \eqref{sub-exponential}
where the optimal size is $N \sim n$ and the weakest second moment assumption \eqref{second moment}
where the optimal size is $N \sim n \log n$. It would be useful to classify the distributions for which 
the logarithmic oversampling is needed. The picture is far from complete -- the uniform distributions
on convex bodies in $\R^n$ for which we now know that the logarithmic oversampling is not needed
are very far from the uniform distributions on $O(n)$ points for which the logarithmic oversampling is needed. 
We conjecture that the logarithmic oversampling is not needed for all distributions with $q$-th moment
with appropriate absolute constant $q$; probably $q = 4$ suffices or even any $q>2$. We will thus assume
that

\begin{equation}							\label{qth moment}
\|X\|_2 \le K \sqrt{n} \text{ a.s.,} \quad 
\E |\< X, x\> | ^q \le L^q 
	\quad \text{for } x \in S^{n-1}.
\end{equation}

\begin{conjecture}
  Let $X$ be a random vector in $\R^n$ that satisfies the moment assumption 
  \eqref{qth moment} for some appropriate absolute constant $q$ and some $K$, $L$. Let $\e > 0$. 
  Then, with high probability, the sample size $N \gtrsim_{K,L,\e} n$ 
  suffices to approximate the covariance matrix $\Sigma$ of $X$ 
  by the sample covariance matrix $\Sigma_N$
  in the operator norm: $\|\Sigma - \Sigma_N\| \le \e$.
\end{conjecture}

In this paper we prove the Conjecture up to an iterated logarithmic factor.

\begin{theorem}											\label{main}
  Consider a random vector $X$ in $\R^n$ ($n \ge 4$) which satisfies moment assumptions \eqref{qth moment} 
  for some $q>4$ and some $K$, $L$. Let $\d > 0$. Then, with probability at least $1-\d$,
  the covariance matrix $\Sigma$ of $X$ can be approximated by the sample covariance matrix
  $\Sigma_N$ as 
  $$
  \|\Sigma - \Sigma_N\|
    \lesssim_{q,K,L,\d} (\log \log n)^2 \Big( \frac{n}{N} \Big)^{\frac{1}{2} - \frac{2}{q}}.
  $$ 
\end{theorem}

\begin{remarks}
1. The notation $a \lesssim_{q,K,L,\d} b$ means that $a \le C(q,K,L,\d) b$ where $C(q,K,L,\d)$ depends
only on the parameters $q,K,L,\d$; see Section~\ref{s: notation} for more notation. 
The logarithms are to the base $2$. 
We put the restriction $n \ge 4$ only to ensure that $\log \log n \ge 1$;
Theorem~\ref{main} and other results below clearly hold for dimensions $n = 1,2,3$ even without 
the iterated logarithmic factors. 

2. It follows that for every $\e>0$, 
the desired approximation $\|\Sigma - \Sigma_N\| \le \e$ is guaranteed if the sample 
has size
$$
N \gtrsim_{q,K,L,\d,\e} (\log \log n)^p n 		\quad \text{where } \frac{1}{p} + \frac{1}{q} = \frac{1}{4}.
$$

3. A similar result holds for independent random vectors $X_1, \ldots, X_N$ 
that are not necessarily identically distributed; we will prove this general result in 
Theorem~\ref{main full}. 

4. The boundedness assumption $\|X\|_2 \le K \sqrt{n}$ in \eqref{qth moment} can often be weakened or
even dropped by a simple modification of Theorem~\ref{main}. This happens, for example, if 
$\max_{i \le N} \|X_i\| = O(\sqrt{n})$ holds with high probability, as one can apply Theorem~\ref{main} 
conditionally on this event. We refer the reader to a thorough discussion of the boundedness assumption 
in Section~1.3 of \cite{V}.
\end{remarks}

\subsection{Extreme eigenvalues of sample covariance matrices}
Theorem~\ref{main} can be used to analyze the spectrum of sample covariance matrices $\Sigma_N$.
The case when the random vector $X$ has i.i.d. coordinates is most studied in random matrix theory. 
Suppose that both $N, n \to \infty$ while the aspect ratio $n/N \to \b \in (0,1]$.
If the coordinates of $X$ have unit variance and finite fourth moment, then clearly $\Sigma = I$.
The largest eigenvalue $\lambda_1(\Sigma_N)$ then converges a.s. to $(1 + \sqrt{\b})^2$,
and the smallest eigenvalue $\lambda_n(\Sigma_N)$ converges a.s. to $(1 - \sqrt{\b})^2$,
see \cite{BY}. For more on the extreme eigenvalues in both asymptotic regime ($N,n \to \infty$) and 
non-asymptotic regime ($N,n$ fixed), see \cite{RV ICM}.

Without independence of the coordinates, analyzing the spectrum of sample covariance matrices $\Sigma_N$
becomes significantly harder. Suppose that $\Sigma = I$.
For sub-exponential distributions, i.e. those satisfying \eqref{sub-exponential}, 
it was proved in \cite{ALPT1} that 
$$
1 - O(\sqrt{\b}) 
\le \lambda_n(\Sigma_N)
\le \lambda_1(\Sigma_N)
\le 1 + O(\sqrt{\b}).
$$ 
(A weaker version with extra $\log(1/\b)$ factors was proved earlier by the same authors in \cite{ALPT}.)
Under only finite moment assumption \eqref{qth moment}, Theorem~\ref{main} clearly yields
$$
1 - O(\log \log n) \b^{\frac{1}{2} - \frac{2}{q}}
\le \lambda_n(\Sigma_N)
\le \lambda_1(\Sigma_N)
\le 1 + O(\log \log n) \b^{\frac{1}{2} - \frac{2}{q}}.
$$
Note that for large exponents $q$, the factor $\b^{\frac{1}{2} - \frac{2}{q}}$ becomes close to $\sqrt{\b}$.

\subsection{Norms of random matrices with independent columns}
One can interpret the results of this paper in terms of random matrices with independent columns.
Indeed, consider an $n \times N$ random matrix $A = [X_1,\ldots,X_N]$ whose columns $X_1,\ldots,X_N$ 
are drawn independently from some distribution on $\R^n$. The sample covariance 
matrix of this distribution is simply $\Sigma_N = \frac{1}{N}AA^T$, so the eigenvalues of $N^{1/2} \Sigma_N$
are the singular values of $A$. In particular, under the same finite moment assumptions 
as in Theorem~\ref{main}, we obtain the bound on the operator norm 
\begin{equation}										\label{norm intro}
\|A\| \lesssim_{q,K,L,\d} \log \log N \cdot (\sqrt{n} + \sqrt{N}).
\end{equation}
This follows from a result leading to Theorem~\ref{main}, see Corollary~\ref{independent columns}.
The bound is optimal up to the $\log \log N$ factor for matrices with i.i.d. entries, 
because the operator norm is 
bounded below by the Euclidean norm of any column and any row.
For random matrices with independent entries, estimate \eqref{norm intro} 
follows (under the fourth moment assumption)
from more general bounds by Seginer \cite{Se} and Latala \cite{La}, and even without the $\log \log N$ factor. 
Without independence of entries, this bound was proved by the author \cite{V products}
for products of random matrices with independent entries and deterministic matrices, 
and also without the $\log \log N$ factor.

\subsection{Organization of the rest of the paper}
In the beginning of Section~\ref{s: outline} we outline the heuristics of our argument.
We emphasize its two main ingredients -- structure of divergent series 
and a decoupling principle. We finish that section with some preliminary
material -- notation (Section~\ref{s: notation}), a known argument 
that solves the approximation problem for sub-gaussian distributions
(Section~\ref{s: subgaussian}), and the previous weaker result of the author \cite{V}
on the approximation problem in the weak $\ell_2$ norm (Section~\ref{s: weak ell2}).

The heart of the paper are Sections~\ref{s: structure} and \ref{s: decoupling}.
In Section~\ref{s: structure} we study the structure of series 
that diverge faster than the iterated logarithm. 
This structure is used in Section~\ref{s: decoupling} to deduce a decoupling principle.
In Section~\ref{s: random matrices} we apply the decoupling principle
to norms of random matrices. 
Specifically, in Theorem~\ref{norm} we estimate the norm of $\sum_{i \in E} X_i \otimes X_i$
uniformly over subsets $E$. We interpret this in Corollary~\ref{independent columns}
as a norm estimate for random matrices with independent columns.
In Section~\ref{s: approximation}, we deduce the general form of our main result 
on approximation of covariance matrices, Theorem~\ref{main full}.

\subsection*{Acknowledgement}
The author is grateful to the referee for useful suggestions.

\section{Outline of the method and preliminaries}				\label{s: outline}

Let us now outline the two main ingredients of our method, which are a new structure theorem for divergent series 
and a new decoupling principle. For the sake of simplicity in this discussion, we shall now concentrate on proving 
the weaker upper bound $\|\Sigma_N\| = O(1)$ in the case $N = n$. Once this simpler case is understood,
the full Theorem~\ref{main} will require a little extra effort using a now standard truncation argument 
due to J.~Bourgain \cite{Bo}.
We thus consider independent copies $X_1,\ldots,X_n$ of a random vector $X$ in $\R^n$ satisfying the 
finite moment assumptions \eqref{qth moment}. We would like 
to show with high probability that 
$$
\| \Sigma_n \| = \sup_{x \in S^{n-1}} \frac{1}{n} \sum_{i=1}^n \< X_i, x\> ^2
= O(1).
$$
In this expression we may recongize a stochastic process indexed by vectors $x$ on the sphere. 
For each fixed $x$, 
we have to control the sum of independent random variables $\sum_i \< X_i, x \> ^2$ with finite moments. 
Suppose the bad event occurs -- for some $x$, this sum is significantly larger than $n$. 
Unfortunately, because of the heavy tails of these random variables, the bad event may occur with polynomial 
rather than exponential probability $n^{-O(1)}$. This is too weak to control these sums for 
all $x$ simultaneously on the $n$-dimensional sphere, where $\e$-nets have exponential sizes in $n$.
So, instead of working with sums of independent random variables, we try to locate some structure in the
summands responsible for the largeness of the sum. 

\subsection{Structure of divergent series}				\label{s: structure intro}
More generally, we shall study the structure of divergent series 
$\sum_i b_i = \infty$, where $b_i \ge 0$.
Let us first suppose that the series diverges faster than logarithmic function, thus 
$$
\sum_{i=1}^n b_i \gg \log n
\quad \text{for some $n \ge 2$}.
$$
Comparing with the harmonic series we see that the non-increasing rearrangement $b^*_i$ 
of the coefficients at some point must be large:
$$
b^*_{n_1} \gg 1/{n_1} 
\quad \text{for some } n_1 \le n.
$$ 
In other words, one can find $n_1$ large terms of the sum: 
there exists an index set $I \subset [n]$ of size $|I| = n_1$ and such that 
$b_i \gg 1/n_1$ for $i \in I$.
This collection of large terms $(b_i)_{i \in I}$ forms a desired structure responsible for the largeness 
of the series $\sum_i b_i$.
Such a structure is well suited to our applications where $b_i$ are independent random variables,
$b_i = \< X_i, x\> ^2 / n$. Indeed, the events $\{ b_i \gg 1/n_1 \}$ are independent, 
and the probability of each such event is easily controlled by finite moment assumptions
\eqref{moment assumptions} through Markov's inequality. This line was developed 
in \cite{V}, but it clearly leads to a loss of logarithmic factor which we are trying to avoid in 
the present paper. 

We will work on the next level of precision, thus studying the structure of
series that diverge slower than the logarithmic function but faster than the iterated logarithm. 
So let us assume that
$$
b^*_i \lesssim 1/i \text{ for all $i$}; \quad 
\sum_{i=1}^n b_i \gg \log \log n
\text{ for some $n \ge 4$}.
$$ 
In Proposition~\ref{structure} we will locate almost the same structure as 
we had for logarithmically divergent series, except up to some factor 
$\log \log n \ll l \lesssim  \log n$, as follows.
For some $n_1 \le n$ there exists an index set $I \subset [n]$ of size $|I| = n_1$, 
such that 
$$
b_i \gg \frac{1}{l n_1} \text{ for } i \in I, \quad
\text{and moreover }
\frac{n}{n_1} \ge 2^{l/2}.
$$

\subsection{Decoupling}						\label{s: decoupling intro}
The structure that we found is well suited to our application where $b_i$ are independent random variables
$b_i = \< X_i, x\> ^2 / n$. In this case we have
\begin{equation}							\label{large on I}
\< X_i, x\> ^2 \gg \frac{n}{l n_1} 
\gtrsim \frac{n}{n_1} \big/ \log \Big( \frac{2n}{n_1} \Big) 
\gtrsim (n/n_1)^{1-o(1)}
\quad \text{for } i \in I.
\end{equation}
The probability that this happens is again easy to control using independence of $\< X_i, x\> $
for fixed $x$, finite moment assumptions \eqref{moment assumptions} and Markov's inequality.
Since there are $\binom{n}{n_1}$ number of ways to choose the subset $I$, 
the probability of the event in \eqref{large on I} is bounded by
$$
\binom{n}{n_1} \P \big\{ \< X_i, x \> ^2 \gg (n/n_1)^{1-o(1)} \big\}^{n_1}
\le \binom{n}{n_1} (10n/n_1)^{-(1-o(1))q/2}
\ll e^{-n_1}
$$
where the last inequality follows because
$\binom{n}{n_1} \le (en/n_1)^{n_1}$ and since $q > 2$.

Our next task is to unfix $x \in S^{n-1}$. The exponential probability estimate we obtained 
allows us to take the union bound over all $x$ in the unit sphere 
of any fixed $n_1$-dimensional subspace, since this sphere has an $\e$-net of size exponential in $n_1$.
We can indeed assume without loss of generality that the vector $x$ in our structural event 
\eqref{large on I} lies in the span of $(X_i)_{i \in I}$ which is $n_1$-dimensional; 
this can be done by projecting $x$ onto this span if necessary.
Unfortunately, this obviously makes $x$ depend on the random vectors $(X_i)_{i \in I}$
and destroys the independence of random variables $ \< X_i, x \> $. 
This hurdle calls for a decoupling mechanism,
which would make $x$ in the structural event \eqref{large on I} 
depend on some small fraction of the vectors $(X_i)_{i \in I}$. 
One would then condition on this fraction of random vectors
and use the structural event \eqref{large on I} for the other half, 
which would quickly lead to completion of the argument.

Our decoupling principle, Proposition~\ref{decoupling}, is a deterministic statement that 
works for fixed vectors $X_i$.
Loosely speaking, we assume that the structural event \eqref{large on I} holds for some $x$
in the span of $(X_i)_{i \in I}$, and we would like to force $x$
to lie in the span of a small fraction of these $X_i$. 
We write $x$ as a linear combination $x = \sum_{i \in I} c_i X_i$. The first step of decoupling is to remove 
the ``diagonal'' term $c_i X_i$ from this sum, while retaining the largeness of $\< X_i, x \> $.
This task turns out to be somewhat difficult, and it will force us to refine our structural result for divergent series
by adding a domination ingredient into it. This will be done at the cost of another $\log \log n$ factor.
After the diagonal term is removed, the number of terms in the sum for $x$ will be reduced 
by a probabilistic selection using Maurey's empirical method.

\subsection{Notation and preliminaries}				\label{s: notation}
We will use the following notation throughout this paper. 
$C$ and $c$ will stand for positive absolute constants;
$C_p$ will denote a quantity which only depends on the parameter $p$,
and similar notation will be used with more than one parameter.
For positive numbers $a$ and $b$, the asymptotic inequality 
$a \lesssim b$ means that $a \le C b$.
Similarly, inequalities of the form $a \lesssim_{p,q} b$ 
mean that $a \lesssim C_{p,q} b$.
Intervals of integers will be denoted by $[n] := \{1,\ldots,\lceil n \rceil\}$ for $n \ge 0$.
The cardinality of a finite set $I$ is denoted by $|I|$.
All logarithms will be to the base $2$.

The non-increasing rearrangement of a finite or infinite sequence of numbers 
$a = (a_i)$ will be denoted by $(a^*_i)$. 
Recall that the $\ell_p$ norm is defined as 
$\|a\|_p = (\sum_i |a_i|^p )^{1/p}$ for $1 \le p < \infty$, and 
$\|a\|_\infty = \max_{i} |a_i|$.
We will also consider the weak $\ell_p$ norm for $1 \le p < \infty$, which is defined 
as the infimum of positive numbers $M$ for which the non-increasing rearrangement $(|a|^*_i)$ 
of the sequence $(|a_i|)$ satisfies $|a|^*_i \le M i^{-1/p}$ for all $i$. 
For sequences of finite length $n$, it follows from definition that the weak $\ell_p$ norm
is equivalent to the $\ell_p$ norm up to a $O(\log n)$ factor, thus
$\|a\|_{p,\infty} \le \|a\|_p \lesssim \log n \cdot \|a\|_{p,\infty}$
for $a \in \R^n$.

In this paper we deal with the $\ell_2 \to \ell_2$ operator norm of $n \times n$
matrices $\|A\|$, also known as spectral norm. By definition, 
$$
\|A\| = \sup_{x \in S^{n-1}} \|Ax\|_2
$$
where $S^{n-1}$ denotes the unit Euclidean sphere in $\R^n$. 
Equivalently, $\|A\|$ is the largest singular value of $A$ and
the largest eigenvalue of $\sqrt{AA^T}$. We will frequently use that for Hermitian
matrices $A$ one has
$$
\|A\| = \sup_{x \in S^{n-1}} |\< Ax, x \> |.
$$

\smallskip
It will be convenient to work in a slightly more general than in Theorem~\ref{main}, 
and consider independent random vectors $X_i$ in $\R^n$
that are not necessarily identically distributed. All we need is that 
moment assumptions \eqref{qth moment} hold uniformly for all vectors:
\begin{equation}							\label{moment assumptions}
  \|X_i\|_2 \le K \sqrt{n} \text{ a.s.},  \quad
  (\E | \< X_i, x\> |^q)^{1/q} \le L \text{ for all } x \in S^{n-1}.
\end{equation}
We can view our goal as establishing a {\em law of large numbers in the operator norm},
and with quantitative estimates on convergence. Thus we would like to show that 
the approximation error  
\begin{equation}							\label{norm is sup}
  \Big\| \frac{1}{N} \sum_{i=1}^N X_i \otimes X_i - \E X_i \otimes X_i \Big\| 
  = \sup_{x \in S^{n-1}} \Big| \frac{1}{N} \sum_{i=1}^N \< X_i, x\> ^2 - \E \< X_i, x\> ^2 \Big|
\end{equation}
is small like in Theorem~\ref{main}.

\subsection{\bf Sub-gaussian distributions}				\label{s: subgaussian}
A solution to the approximation problem is well known and easy for 
sub-gaussian random vectors, those satisfying \eqref{sub-gaussian}. 
The optimal sample size here is proportional to the dimension, thus
$N = O_{L,\e} (n)$. For the reader's convenience, we recall and prove 
a general form of this result.

\begin{proposition}[Sub-gaussian distributions]						\label{subgaussian prop}
  Consider independent random vectors $X_1,\ldots,X_N$ in $\R^n$, $N \ge n$, 
  which have sub-gaussian distribution as in \eqref{sub-gaussian} for some $L$.
  Then for every $\d > 0$ with probability at least $1-\d$ one has
   $$
   \Big\| \frac{1}{N} \sum_{i=1}^N X_i \otimes X_i - \E X_i \otimes X_i \Big\| 
    \lesssim_{L,\d} \Big( \frac{n}{N} \Big)^{\frac{1}{2}}.
   $$  
\end{proposition}
One should compare this with our main result, Theorem~\ref{main}, 
which yields almost the same conclusion
under only finite moment assumptions on the distribution, 
except for an iterated logarithmic factor and a slight loss of the exponent $1/2$ 
(the latter may be inevitable when dealing with finite moments).

\smallskip

The well known proof of Proposition~\ref{subgaussian prop} is based on Bernstein's deviation 
inequality for independent random variables and an $\e$-net argument. 
The latter allows to replace the sphere $S^{n-1}$ in the computation of the
norm in \eqref{norm is sup} by a finite $\e$-net as follows.

\begin{lemma}[Computing norms on $\e$-nets]						\label{nets}
  Let $A$ be a Hermitian $n \times n$ matrix,
  and let $\NN_\e$ be an $\e$-net of the unit Euclidean sphere $S^{n-1}$
  for some $\e \in [0,1)$. Then 
  $$
  \|A\| = \sup_{x \in S^{n-1}} |\< Ax, x\> | 
  \le (1 - 2\e)^{-1} \sup_{x \in \NN_\e} |\< Ax, x\> |.
  $$  
\end{lemma}

\begin{proof}
Let us choose $x \in S^{n-1}$ for which $\|A\| = |\< Ax, x\> |$,
and choose $y \in \NN_\e$ which approximates $x$ as $\|x - y\|_2 \le \e$.
It follows by the triangle inequality that
\begin{align*}							
|\< Ax, x\> - \< Ay, y\> |
&= |\< Ax, x-y\> + \< A(x-y), y\> |\\
&\le \|A\| \|x\|_2 \|x-y\|_2 + \|A\| \|x-y\|_2 \|y\|_2
\le 2 \|A\| \e.
\end{align*}
It follows that 
$$
|\< Ay, y \> | \ge |\< Ax, x\> | - 2 \|A\| \e = (1-2\e) \|A\|.
$$
This completes the proof.
\end{proof}

\medskip
\begin{proof}[Proof of Proposition~\ref{subgaussian prop}.]
Without loss of generality,
 we can assume that in the sub-gaussian assumption \eqref{sub-gaussian} we have $L=1$ 
by replacing $X_i$ by $X_i/L$.
Identity \eqref{norm is sup} expresses the norm in question as a supremum 
over the unit sphere $S^{n-1}$. Next, Lemma~\ref{nets} allows  
to replace the sphere in \eqref{norm is sup} by its $1/2$-net $\NN$ 
at the cost of an absolute constant factor.
Moreover, we can arrange so that the net has size $|\NN| \le 6^n$; 
this follows by a standard volumetric argument (see \cite[Lemma~9.5]{LT}).

Let us fix $x \in \NN$.
The sub-gaussian assumption on $X_i$ implies that 
the random variables $\< X_i, x\> ^2$ are sub-exponential: 
$\P (\< X_i, x\> ^2 > t) \le 2 e^{-t}$ for $t > 0$.
Bernstein's deviation inequality for independent sub-exponential random variables
(see e.g. \cite[Section 2.2.2]{VW}) yields for all $\e > 0$ that 
\begin{equation}							\label{deviation subexponential}
\P \big\{ \big| \frac{1}{N} \sum_{i=1}^N \< X_i, x\> ^2 - \E \< X_i, x\> ^2 \big| > \e \big\} 
\le 2 e^{-c \e^2 N}.
\end{equation}
Now we unfix $x$. Using \eqref{deviation subexponential} for each $x$ in the net $\NN$, we conclude by 
the union bound that the event 
$$
\big| \frac{1}{N} \sum_{i=1}^N \< X_i, x\> ^2 - \E \< X_i, x\> ^2 \big| < \e
\quad \text{for all } x \in \NN
$$
holds with probability at least 
$$
1 - |\NN| \cdot 2 e^{-c \e^2 N} 
\ge 1 - 2 e^{2n - c \e^2 N}.
$$
Now if we choose $\e^2 = (4/c) \log(2/\d) \, n/N$, this probability is 
further bounded below by $1-\d$ as required.
By the reduction from the sphere to the net 
mentioned in the beginning of the argument, this completes the proof. 
\end{proof}

\subsection{Results in the weak $\ell_2$ norm, and almost orthogonality of $X_i$}	\label{s: weak ell2}
A truncation argument of J.~Bourgain \cite{Bo} reduces the approximation 
problem to finding an upper bound on 
$$
\Big\| \sum_{i \in E} X_i \otimes X_i \Big\| 
= \sup_{x \in S^{n-1}} \sum_{i \in E} \< X_i, x\> ^2 
= \sup_{x \in S^{n-1}} \|( \< X_i, x\> )_{i \in E}\|_2^2
$$
uniformly for all index sets $E \subset [N]$ with given size.
A weaker form of this problem, with the weak $\ell_2$ norm of the sequence $\< X_i, x\> $
instead of the its $\ell_2$ norm, was studied in \cite{V}.
The following bound was proved there:

\begin{theorem}[\cite{V} Theorem~3.1]							\label{weak ell2 bound}
  Consider random vectors $X_1,\ldots,X_N$ which satisfy moment assumptions
  \eqref{moment assumptions} for some $q > 4$ and some $K$, $L$.
  Then, for every $t \ge 1$, with probability at least $1 - C t^{-0.9 q}$ one has
  $$
  \sup_{x \in S^{n-1}} \|( \< X_i, x\> )_{i \in E}\|_{2,\infty}^2 
    \lesssim_{q,K,L} n + t^2 \Big( \frac{N}{|E|} \Big)^{4/q} |E|
    \quad \text{for all } E \subseteq [N]. 
  $$
\end{theorem}

For most part of our argument (through decoupling), we treat $X_i$ as fixed non-random vectors.
The only property we require from $X_i$ is that they are almost pairwise orthogonal.
For random vectors, an almost pairwise orthogonality easily follows 
from the moment assumptions \eqref{moment assumptions}:

\begin{lemma}[\cite{V} Lemma~3.3]							\label{almost orthogonality}
  Consider random vectors $X_1,\ldots,X_N$ which satisfy moment assumptions
  \eqref{moment assumptions} for some $q > 4$ and some $K$, $L$.
  Then, for every $t \ge 1$,  with probability at least $1 - C t^{-q}$ one has
  \begin{equation}							\label{cross terms OK}
    \frac{1}{|E|} \sum_{i \in E,\; i \ne k} \< X_i, X_k \> ^2
    \lesssim_{q,K,L} t^2 \Big( \frac{N}{|E|} \Big)^{4/q} n
    \quad \text{for all } E \subseteq [N], \; k \in [N].
  \end{equation}
\end{lemma}

\section{Structure of divergent series}				\label{s: structure}

In this section we study the structure of series which diverge slower than the logarithmic
function but faster than an iterated logarithm. This is summarized in the following result.

\begin{proposition}[Structure of divergent series]						\label{structure}
  Let $\a \in (0,1)$. 
  Consider a vector $b = (b_1,\ldots,b_m) \in \R^m$ ($m \ge 4$) that satisfies
  \begin{equation}										\label{divergence}
  \|b\|_{1,\infty} \le 1, \quad 
  \|b\|_1 \gtrsim_{\a} (\log \log m)^2.
  \end{equation}
  Then there exist a positive integer $l \le \log m$ and a subset of indices $I_1 \subseteq [m]$
  such that the following holds. 
  Given a vector $\lambda = (\lambda_i)_{i \in I_1}$ such that  $\|\lambda\|_1 \le 1$, 
  one can find a further subset $I_2 \subseteq I_1$ with the following two properties.
  
  (i) (Regularity): 
  the sizes $n_1 := |I_1|$ and $n_2 := |I_2|$ satisfy 
  $$
  2^{l/2} \le \frac{m}{n_1} \le \frac{m}{n_2} \le \Big( \frac{m}{n_1} \Big)^{1+\a}.
  $$
  
  (ii) (Largeness of coefficients):
  \begin{align*}
  &|b_i| \ge \frac{1}{l n_1} \text{ for $i \in I_1$}; \\
  &|b_i| \ge \frac{1}{l n_2} \text{ and } |b_i| \ge 2 |\lambda_i| \text{ for $i \in I_2$}.
  \end{align*}
  
Furthermore, we can make $l \ge C_\a \log \log m$ with arbitrarily large $C_\a$
by making the dependence on $\a$ implicit in the assumption \eqref{divergence} sufficiently large.
\end{proposition}

\begin{remarks}
  1. Proposition~\ref{structure} is somewhat nontrivial even if one ignores the vector $\lambda$ 
  and the further subset $I_2$.
  In this simpler form the result was introduced informally in Section~\ref{s: structure intro}.
  The structure that we find is located in the coefficients $b_i$ on the index set $I_1$. 
  Note that the largeness condition (ii) for these coefficients is easy to prove if we disregard the regularity 
  condition (i). 
  Indeed, since $\|b\|_{1,\infty} \gtrsim (\log m)^{-1} \|b\|_1 \gg 1/\log m$, 
  we can choose $l = \log m$ and obtain a set $I_1$ satisfying (ii) by the definition of the weak $\ell_2$ norm.
  But the regularity condition (i) guarantees the smaller level 
  $$
  l \lesssim \log (m/n_1)
  $$
  which will be crucial in our application to decoupling.
  
  2. The freedom to choose $\lambda$ in Proposition~\ref{structure}
  ensures that the structure located in the set $I_1$
  is in a sense hereditary; it can pass to subsets $I_2$. The domination of $\lambda$
  by $b$ on $I_2$ will be crucial in the removal of the diagonal terms in our application 
  to decoupling.
\end{remarks}

We now turn to the proof of Proposition~\ref{structure}.
Heuristically, we will first find many (namely, $l$) 
sets $I_1$ on which the coefficients are large as in (ii),
then choose one that satisfies the regularity condition (i).
This regularization step will rely on the following elementary lemma.

\begin{lemma}[Regularization]				\label{regularization}
  Let $N$ be a positive integer. 
  Consider a nonempty subset $J \subset [L]$ with size $l := |J|$.
  Then, for every $\a \in (0,1)$,
  there exist elements $j_1, j_2 \in J$ that satisfy the following two properties.
  
  (i) (Regularity): 
  $$
  l/2 \le j_1 \le j_2 \le (1+\a) j_1.
  $$
  
  (ii) (Density):
  $$
  |J \cap [j_1, j_2]| \gtrsim_\a \frac{l}{\log(2L/l)}.
  $$
\end{lemma}

\begin{proof}
We will find $j_1$, $j_2$ as some consecutive terms of the following geometric progression. 
Define $j^{(0)} \in J$ to be the (unique) element such that 
$$
|J \cap [1, j^{(0)}]| = \lceil l/2 \rceil, \quad
\text{and let } j^{(k)} := (1+\a) j^{(k-1)}, 
\quad k=1,2,\ldots.
$$
We will only need to consider $K$ terms of this progression, where 
$K := \min \{k :\; j^{(k)} \ge L \}$.
Since $j^{(0)} \ge \lceil l/2 \rceil \ge l/2$, we have 
$j^{(k)} \ge (1+\a)^k j^{(0)} \ge (1+\a)^k l/2$. On the other hand, $j^{(K-1)} \le L$. 
It follows that $K \lesssim_\a \log(2L/l)$.

We claim that there exists a term $1 \le k \le K$ such that 
\begin{equation}									\label{jk-1 jk}
  |J \cap [j^{(k-1)}, j^{(k)}]| \ge \frac{l}{3K} \gtrsim_\a  \frac{l}{\log(2L/l)}.
\end{equation}
Indeed, otherwise we would have 
$$
l = |J| \le |J \cap [1, j^{(0)}]| + \sum_{k=1}^K |J \cap [j^{(k-1)}, j^{(k)}]|
\le \Big\lceil \frac{l}{2} \Big\rceil + K \cdot \frac{l}{3K}
< l,
$$
which is impossible.

The terms $j_1 := j^{(k-1)}$ and $j_2 := j^{(k)}$ for which \eqref{jk-1 jk} holds
clearly satisfy (i) and (ii) of the conclusion. By increasing $j_1$
and decreasing $j_2$ if necessary we can assume that $j_1, j_2 \in J$. 
This completes the proof.
\end{proof}

\medskip

\begin{proof}[Proof of Proposition~\ref{structure}]
We shall prove the following slightly stronger statement. 
Consider a sufficiently large number 
$$
K \gtrsim_\a \log \log m.
$$
Assume that the vector $b$ satisfies 
$$
\|b\|_{1,\infty} \le 1, \quad 
\|b\|_1 \gtrsim_{\a} K \log \log m.
$$
We shall prove that the conclusion of the Proposition holds 
with (ii) replaced by:
  \begin{align}
    &|b_i| \ge \frac{K}{2 l n_1} \text{ for $i \in I_1$}; 					\label{bi on I1}\\
    &|b_i| \ge \frac{K}{2 l n_2} \ge 2 |\lambda_i| \text{ for $i \in I_2$}	\label{bi on I2}.
  \end{align}

We will construct $I_1$ and $I_2$ in the following way. First we decompose the index set $[m]$ 
into blocks $\Omega_1, \ldots, \Omega_L$ on which the coefficients $b_i$ have similar magnitude; 
this is possible with $L \sim \log m$ blocks. Using the assumption $\|b\|_1 \gtrsim (\log \log m)^2$, 
one easily checks that many (at least $l \sim \log \log m$) of the blocks $\Omega_j$ have large contribution 
(at least $1/j$) to the sum $\|b\|_1 = \sum_i |b_i|$. We will only focus on such large blocks in the rest of the argument.
At this point, the union of these blocks could be declared $I_1$. We indeed proceed this way, except we first
use Regularization Lemma~\ref{regularization} on these blocks in order to obtain the required regularity property (ii).
Finally, assume we are given coefficients $(\lambda_i)_{i \in I_1}$ with small sum $\sum |\lambda_i| \le 1$
as in the assumption. Since the coefficients $b_i$ are large on $I_1$ by construction, the pigeonhole principle
will yield (loosely speaking) a whole block of coefficients $\Omega_j$ where $b_i$ will dominate as required,
$|b_i| \ge 2 |\lambda_i|$. We declare this block $I_2$ and complete the proof. 
Now we pass to the details of the argument.

\smallskip
{\bf Step 1: decomposition of $[m]$ into blocks.}
Without loss of generality,
$$
\frac{1}{m} < b_i \le 1, \quad
\lambda_i \ge 0 \quad
\text{for all } i \in [m].
$$
Indeed, we can clearly assume that $b_i \ge 0$ and $\lambda_i \ge 0$. 
The estimate $b_i \le 1$ follows from the assumption: $\|b\|_\infty \le \|b\|_{1,\infty} \le 1$.
Furthermore, the contribution of the small coefficients $b_i \le 1/m$ to the norm 
$\|b\|_1$ is at most $1$, while by the assumption $\|b\|_1 \gtrsim_\a K \log \log m \ge 2$. 
Hence we can ignore these small coefficients by replacing $[m]$ with the subset 
corresponding to the coefficients $b_i \ge 1/m$.

We decompose $[m]$ into disjoint subsets (which we call blocks) according to the magnitude of $b_i$, 
and we consider the contribution of each block $\Omega_j$ to the norm $\|b\|_1$:
$$
\Omega_j := \big\{ i \in [m] :\; 2^{-j} < b_i \le 2^{-j+1} \big\}; 
\quad m_j := |\Omega_j|;
\quad B_j := \sum_{i \in \Omega_j} b_i.
$$ 
By our assumptions on $b$, there are at most $\log m$ nonempty blocks $\Omega_j$.
As $\|b\|_{1,\infty} \le 1$, Markov's inequality yields for all $j$ that
\begin{gather}								
  m_j \le \sum_{k \le j} m_k = \big| \big\{ i \in [m] :\; b_i > 2^{-j} \big\} \big| \le 2^j; 	 \label{size mj}\\							
  B_j \le m_j 2^{-j+1} \le 2.  											 \label{Bj small}
\end{gather}

Only the blocks with large contributions $B_j$ will be of interest to us. Their number is
$$
l := \max \big\{ j \in [\log m] :\; B^*_j \ge K/j \big\};
$$
and we let $l=0$ if it happens that all $B_j < K/j$.
We claim that there are many such blocks:
\begin{equation}													\label{l large}
  \frac{1}{5} K \log \log m \le l \le \log m.
\end{equation}
Indeed, by the assumption and using \eqref{Bj small} we can bound
$$
K \log \log m \le \|b\|_1 = \sum_{j=1}^{\log m} B^*_j
\le 2l + 0.6 K \log \log m,
$$
which yields \eqref{l large}.

\smallskip
{\bf Step 2: construction of the set $I_1$.}
As we said before, we are only interested in blocks $\Omega_j$ with large contributions $B_j$. 
We collect the indices of such blocks into the set 
$$
\bar{J} := \big\{ j \in [\log m] :\; B_j \ge K/l \big\}.
$$
Since the definition of $l$ implies that $B^*_l \ge K/l$, we have $|\bar{J}| \ge l$.
Then we can apply Regularization Lemma~\ref{regularization} to the set
$\{ \log m - j :\; j \in \bar{J} \} \subseteq [\log m]$. 
Thus we find two elements $j', j'' \in \bar{J}$ satisfying 
\begin{equation}																				\label{j primes}
  l/2 \le \log m - j' \le \log m - j'' \le (1+\a/2) (\log m - j'),
\end{equation}
and such that the set 
$$
J := \bar{J} \cap [j'', j']
$$
has size $|J| \gtrsim_\a l / \log \log m$. 
Since by our choice of $K$ we can assume that $K \ge 8 \log \log m$,
we obtain
\begin{equation}																				\label{J large}
  |J| \ge \frac{8 l}{K}.
\end{equation}
We are going to show that the set  
$$
I_1 := \bigcup_{j \in J} \Omega_j
$$
satisfies the conclusion of the Proposition.

\smallskip
{\bf Step 3: sizes of the coefficients $b_i$ for $i \in I_1$.}
Let us fix $j \in J \subseteq \bar{J}$.
From the definition of $\bar{J}$ we know that the contribution $B_j$ 
is large: $B_j \ge K/l$. 
One consequence of this is a good estimate of the size $m_j$ of the block $\Omega_j$. 
Indeed, the above bound together with \eqref{Bj small} this implies
\begin{equation}																					\label{mj exp}
  \frac{K}{2 l} 2^j \le m_j \le 2^j			\quad \text{ for } j \in J.
\end{equation}
Another consequence of the lower bound on $B_j$ is the required lower bound
on the individual coefficients $b_i$. Indeed, by construction of $\Omega_j$
the coefficients $b_i$, $i \in \Omega_j$ are within the factor $2$ from each other. 
It follows that
\begin{equation}																					\label{bi large on blocks}
  b_i \ge \frac{1}{2|\Omega_j|} \sum_{i \in \Omega_j} b_i
  = \frac{B_j}{2m_j}
  \ge \frac{K}{2 l m_j}
  \quad \text{for } i \in \Omega_j, \; j \in J.
\end{equation}
In particuar, since by construciton $\Omega_j \subseteq I_1$, we have $m_j \le |I_1|$, 
which implies
$$
b_i \ge \frac{K}{2 l |I_1|} \quad \text{for } i \in I_1.
$$ 
We have thus proved the required lower bound \eqref{bi on I1}.

\smallskip
{\bf Step 4: Construction of the set $I_2$, and sizes of 
  the coefficients $b_i$ for $i \in I_2$.}
Now suppose we are given a vector $\lambda = (\lambda_i)_{i \in I_1}$ with $\|\lambda\|_1 \le 1$.
We will have to construct a subset $I_2 \subset I_1$ as in the conclusion,
and we will do this as follows. 
Consider the contribution of the block $\Omega_j$ to the norm $\|\lambda\|_1$:
$$
L_j := \sum_{i \in \Omega_j} \lambda_j, \quad j \in J.
$$
On the one hand, the sum of all contributions is bounded as 
$\sum_{j \in J} L_j = \|\lambda\|_1 \le 1$.
On the other hand, there are many terms in this sum: $|J| \ge 8l/K$
as we know from \eqref{J large}.
Therefore, by the pigeonhole principle some of the contributions must be small: 
there exists $j_0 \in J$ such that 
$$
L_{j_0} \le K/8l.
$$

This in turn implies via Markov's inequality that most of the coefficients 
$\lambda_i$ for $i \in \Omega_{j_0}$ are small, and we shall declare these set of indices $I_2$. 
Specifically, since 
$L_{j_0} = \sum_{i \in \Omega_{j_0}} \lambda_j \le K/8l$ and $|\Omega_{j_0}| = m_{j_0}$, 
using Markov's inequality we see that the set
$$
I_2 := \big\{ i \in \Omega_{j_0} :\; \lambda_i \le \frac{K}{4 l m_{j_0}} \big\}
$$
has cardinality 
\begin{equation}																\label{I2 mj0}
  \frac{1}{2} m_{j_0} \le |I_2| \le m_{j_0}.
\end{equation}
Moreover, using \eqref{bi large on blocks}, we obtain
$$
b_i \ge \frac{K}{2 l m_{j_0}} \ge 2 \lambda_i
\quad \text{for } i \in I_2.
$$
We have thus proved the required lower bound \eqref{bi on I2}.

\smallskip
{\bf Step 5: the sizes of the sets $I_1$ and $I_2$.}
It remains to check the regularity property (i) of the conclusion of the Proposition.
We bound
\begin{align*}
|I_1| 
  &= \sum_{j \in J} m_j	 \quad \text{(by definition of $I_1$)} \\
  &\le \sum_{j \le j'} m_j  \quad \text{(by definition of $J$)} \\
  &\le 2^j.  \quad \text{(by \eqref{size mj})}
\end{align*}
Therefore, using \eqref{j primes} we conclude that 
\begin{equation}																\label{mj over I1}
  \frac{m}{|I_1|} \ge 2^{\log m - j'} \ge 2^{l/2}.
\end{equation}
We have thus proved the first inequality in (i) of the conclusion of the Proposition.
Similarly, we bound
\begin{align*}
|I_2|
  &\ge \frac{1}{2} m_{j_0}				\quad \text{(by \eqref{I2 mj0})} \\
  &\ge \frac{K}{4 l} 2^{j_0}			\quad \text{(by \eqref{mj exp}, and since $j_0 \in J$)} \\
  &\ge \frac{K}{4 l} 2^{j''}.			\quad \text{(by definition of $J$, and since $j_0 \in J$)} \\  
\end{align*}
Therefore
\begin{align*}
\frac{m}{|I_2|} 
  &\le \frac{4 l}{K} 2^{\log m - j''}  \\ 
  &\le \frac{8}{K} (\log m -j') 2^{(1+\a/2)(\log m - j')}		\quad \text{(by \eqref{j primes})}\\
  &\le \frac{\a}{2} (\log m -j') 2^{(1+\a/2)(\log m - j')}	\quad \text{(by the assumption on $K$)}\\
  &\le 2^{(1+\a)(\log m - j')}\\
  &\le \Big( \frac{m}{|I_1|} \Big)^{1+\a}.										\quad \text{(by \eqref{mj over I1})}
\end{align*}
This completes the proof of (i) of the conclusion, and of the whole Proposition~\ref{structure}.
\end{proof}

\section{Decoupling}					\label{s: decoupling}

In this section we develop a decoupling principle, which was informally introduced in 
Section~\ref{s: decoupling intro}. In contrast to other decoupling results used in probabilistic contexts, 
our decoupling principle is non-random. It is valid for 
arbitrary fixed vectors $X_i$ which are almost pairwise orthogonal as in \eqref{cross terms OK}.
An example of such vectors are random vectors, as we observed earlier 
in Lemma~\ref{almost orthogonality}. Thus in this section we will consider vectors 
$X_1,\ldots,X_m \in \R^n$
that satisfy the following almost pairwise orthogonality assumptions for some 
$r' \ge 1, K_1, K_2$:

\begin{equation}														\label{weak orthonormality}
\begin{split}
  &\|X_i\|_2 \le K_1 \sqrt{n}; \\
  &\frac{1}{|E|} \sum_{i \in E,\; i \ne k} \< X_i, X_k \> ^2
    \le K_2^4 \Big( \frac{N}{|E|} \Big)^{1/r'} n
    \quad \text{for all } E \subseteq [m], \; k \in [m].
\end{split}
\end{equation}

In the earlier work \cite{V} we developed a weaker decoupling principle,
which was valid for the weak $\ell_2$ norm instead of $\ell_2$ norm. Let us recall this result first. 
Assume that for vectors $X_i$ satisfying \eqref{weak orthonormality} with $r'=r$ one has
$$
\sup_{x \in S^{n-1}} \|( \< X_i, x\> )_{i=1}^m\|_{2,\infty}^2 
\gtrsim_{r,K_1,K_2} n + \Big( \frac{N}{m} \Big)^{1/r} m.
$$
Then the Decoupling Proposition~2.1 of \cite{V} implies that 
there exist disjoint sets of indices $I, J \subseteq [m]$ such that $|J| \le \d |I|$, 
and there exists a vector $y \in S^{n-1} \cap \Span(X_j)_{j \in J}$, such that 
$$
\< X_i, y\> ^2 \ge \Big( \frac{N}{|I|} \Big)^{1/r}
\quad \text{for } i \in I.
$$
Results of this type are best suited for applications to random independent vectors $X_i$. 
Indeed, the events that $\< X_i, y\> ^2$ is large are independent for $i \in I$
because $y$ does not depend on $(X_i)_{i \in I}$.
The probability of each such event is easy to bound
using the moment assumptions \eqref{moment assumptions}.

In our new decoupling principle, we replace the weak $\ell_2$ norm by 
the $\ell_2$ norm at the cost of an iterated logarithmic factor and 
a slight loss of the exponent. Our result will thus operate in the regime where
the weak $\ell_2$ norm is small while $\ell_2$ norm is large. 
We summarize this in the following proposition.

\begin{proposition}[Decoupling]										\label{decoupling}
  Let $n \ge 1$ and $4 \le m \le N$ be integers, 
  and let $1 \le r < \min(r', r'')$ and $\d \in (0,1)$.
  Consider vectors $X_1,\ldots, X_m \in \R^n$ which satisfy the 
  weak orthonormality conditions \eqref{weak orthonormality} for some $K_1, K_2$.
  Assume that for some $K_3 \ge \max(K_1,K_2)$ one has
  \begin{gather}	
  \sup_{x \in S^{n-1}} \|( \< X_i, x\> )_{i=1}^m\|_{2,\infty}^2 
    \le K_3^2 \Big[ n + \Big( \frac{N}{m} \Big)^{1/r} m \Big], 					\label{decoupling 2 infty}\\
  \Big\| \sum_{i=1}^m X_i \otimes X_i \Big\| =
  \sup_{x \in S^{n-1}} \sum_{i=1}^m \< X_i, x\> ^2
    \gtrsim_{r,r',r'',\d} K_3^2 (\log \log m)^2 \Big[ n + \Big( \frac{N}{m} \Big)^{1/r} m \Big].	\label{decoupling assumption}
  \end{gather}
  Then there exist nonempty disjoint sets of indices $I, J \subseteq [m]$ such that $|J| \le \d |I|$, 
  and there exists a vector $y \in S^{n-1} \cap \Span(X_j)_{j \in J}$, such that 
  $$
  \< X_i, y\> ^2 \ge K_3^2 \Big( \frac{N}{|I|} \Big)^{1/r''}
  \quad \text{for } i \in I.
  $$
\end{proposition}

The proof of the Decoupling Proposition~\ref{decoupling} will use Proposition~\ref{structure}
in order to locate the structure of the large coefficients $\< X_i, x\> $. 
The following elementary lemma will be used in the argument.

\begin{lemma}										\label{order statistics}
  Consider a vector $\lambda = (\lambda_1,\ldots,\lambda_n) \in \R^n$ 
  which satisfies
  $$
  \|\lambda\|_1 \le 1, \quad \|\lambda\|_\infty \le 1/K
  $$
  for some integer $K$. Then, for every real numbers $(a_1,\ldots,a_n) \in \R^n$ one has
  $$
  \sum_{i=1}^n \lambda_i a_i \le \frac{1}{K} \sum_{i=1}^K a^*_i.
  $$
\end{lemma}

\begin{proof}
It is easy to check that each extreme point of the convex set 
$$
\Lambda := \{ \lambda \in \R^n :\; \|\lambda\|_1 \le 1, \; \|\lambda\|_\infty \le 1/K \}
$$
has exactly $K$ nonzero coefficients which are equal to $\pm 1/K$.
Evaluating the linear form $\sum \lambda_i a_i$ on these extreme points, we obtain
$$
\sup_{\lambda \in \Lambda} \sum_{i=1}^n \lambda_i a_i
= \sup_{\lambda \in \ext(\Lambda)} \sum_{i=1}^n \lambda_i a_i
= \frac{1}{K} \sum_{i=1}^K a^*_i.
$$
The proof is complete.
\end{proof}

\medskip

\begin{proof}[Proof of Decoupling Proposition~\ref{decoupling}.]
By replacing $X_i$ with $X_i/K_3$ we can assume 
without loss of generality that $K_1 = K_2 = K_3 = 1$. 
By perturbing the vectors $X_i$ slightly we may also assume that $X_i$ are all different.

\smallskip
{\bf Step 1: separation and the structure of coefficients.} 
Suppose the assumptions of the Proposition hold, and let us choose a vector $x \in S^{n-1}$ 
which attains the supremum in \eqref{decoupling assumption}. We denote
$$
a_i := \< X_i, x \>  \quad \text{for } i \in [m],
$$
and without loss of generality we may assume that $a_i \ne 0$.
We also denote
$$
\bar{n} := n + \Big( \frac{N}{m} \Big)^{1/r} m.
$$
We choose parameter $\a = \a(r,r',r'',\d) \in (0,1)$ sufficiently small; 
its choice will become clear later on in the argument.
At this point, we may assume that
$$
\|a\|_{2,\infty}^2 \le \bar{n}, \quad
\|a\|_2^2 \gtrsim_\a (\log \log m)^2 \bar{n}.
$$
We can use Structure Proposition~\ref{structure} to locate the structure 
in the coefficients $a_i$.
To this end, we apply this result for $b_i = a_i^2/\bar{n}$
and obtain a number $l \le \log m$ and a 
subset of indices $I_1 \subseteq [m]$.
We can also assume that $l$ is sufficiently large --
larger than an arbitrary quantity which depends on $\a$. 

Since a vector $x \in S^{n-1}$ satisfies 
$\< X_i/a_i, x \> = 1$ for all $i \in I_1$ (in fact for all $i \in [m]$), 
a separation argument for the convex hull $K := \conv(X_i/a_i)_{i \in I_1}$ 
yields the existence of a vector $\bar{x} \in \conv(K \cup 0)$ that satisfies
\begin{equation}  											\label{xi ai x}
  \|\bar{x}\|_2 = 1, \quad 
  \< X_i/a_i, \bar{x} \> \ge 1  \quad \text{for } i \in I_1.
\end{equation}
We express $\bar{x}$ as a convex combination
\begin{equation}											\label{convex combination}
  \bar{x} = \sum_{i \in I_1} \lambda_i X_i/a_i \quad
  \text{for some } \lambda_i \ge 0, \quad
  \sum_{i \in I_1} \lambda_i \le 1.
\end{equation}

We then read the conclusion of Structure Proposition~\ref{structure} as follows. 
There exists a futher subset of indices $I_2 \subseteq I_1$ such that the sizes 
$n_1 := |I_1|$ and $n_2 := |I_2|$ are regular in the sense that
\begin{equation}												\label{n1 n2}
  2^{l/2} \le \frac{m}{n_1} \le \frac{m}{n_2} \le \Big( \frac{m}{n_1} \Big)^{1+\a},
\end{equation}
and the coefficients on $I_1$ and $I_2$ are large:
\begin{align}														
  &a_i^2 \ge \frac{\bar{n}}{l n_1}  &\text{for } i \in I_1,				\label{ai on I1}\\
  &a_i^2 \ge \frac{\bar{n}}{l n_2} \text{ and } a_i^2 \ge 2 \lambda_i \bar{n} 
  	& \text{for } i \in I_2.										\label{ai on I2} 
\end{align}
Furthermore, we can make $l$ sufficiently large depending on $\a$, say
$l \ge 100/\a^2$.

\smallskip
{\bf Step 2: random selection.}
We will reduce the number of terms $n_1$ in the sum \eqref{convex combination} 
defining $\bar{x}$ using random selection, trying to bring this number down to about $n_2$.
As is usual in dealing with sums of independent 
random variables, we will need to ensure that all summands $\lambda_i X_i/a_i$ have controlled magnitudes. 
To this end, we have $\|X_i\|_2 \le \sqrt{n}$ by the assumption, and  
we can bound $1/a_i$ through \eqref{ai on I1}. Finally, we have an a priori bound 
$\lambda_i \le 1$ on the coefficients of the convex combination. However, the latter bound will 
turn out to be too weak, and we will need $\lambda_i \lesssim_\a 1/n_2$ instead. 
To make this happen, instead of the sets $I_1$ and $I_2$ we will be working on their large
subsets $I_1'$ and $I_2'$ defines as 
$$
I_1' := \big\{ i \in I_1 :\; \lambda_i \le \frac{C_\a}{n_2} \big\}, \quad
I_2' := \big\{ i \in I_2 :\; \lambda_i \le \frac{C_\a}{n_2} \big\}
$$
where $C_\a$ is a sufficiently large quantity whose value we will choose later.
By Markov's inequality, this incurs almost no loss of coefficients:
\begin{equation}																		\label{I minus I}
  |I_1 \setminus I_1'| \le \frac{n_2}{C_\a}, \quad
  |I_2 \setminus I_2'| \le \frac{n_2}{C_\a}.
\end{equation}
 
We will perform a random selection on $I_1$ using B.~Maurey's empirical method \cite{Pi}.
Guided by the representation \eqref{convex combination} of $\bar{x}$ as a convex combination, 
we will treat $\lambda_i$ as probabilities, thus introducing a random vector $V$ with distribution
$$
\P \{ V = X_i/a_i \} = \lambda_i, \quad i \in I_1'.
$$ 
On the remainder of the probability space, we assign $V$ zero value:
$\P\{V=0\} = 1 - \sum_{i \in I_1'} \lambda_i$.
Consider independent copies $V_1, V_2, \ldots$ of $V$.
We are not going to do a random selection on the set $I_1 \setminus I_1'$ where the coefficients
$\lambda_i$ may be out of control, 
so we just add its the contribution by defining independent random vectors
$$
Y_j := V_j + \sum_{i \in I_1 \setminus I_1'} \lambda_i X_i a_i, 
\quad j=1,2,\ldots
$$
Finally, for $C_\a' := C_\a/\a$, we consider the average of 
about $n_2$ such vectors: 
\begin{equation}													\label{y bar}
  \bar{y} := \frac{C_\a'}{n_2} \sum_{j=1}^{n_2/C_\a'} Y_j.
\end{equation}

We would like to think of $\bar{y}$ as a random version of the vector $\bar{x}$. 
This is certainly true in expectation:
$$
\E \bar{y} = \E Y_1 = \sum_{i \in I_1} \lambda_i X_i/a_i = \bar{x}.
$$
Also, like $\bar{x}$, the random vector $\bar{y}$ is a convex combination
of terms $X_i/a_i$ (now even with equal weights). 
The advantage of $\bar{y}$ over $\bar{x}$
is that it is a convex combination of much fewer terms, 
as $n_2/C_\a' \ll n_2 \le n_1$.
In the next two steps, we will check that $\bar{y}$ is similar to $\bar{x}$ in the sense
that its norm is also well bounded above, and at least $\sim n_2$ of the inner products 
$\< X_i/a_i, \bar{y} \> $ are still nicely bounded below. 

\smallskip
{\bf Step 3: control of the norm.}
By independence, we have 
\begin{align*}
\E \|\bar{y} - \bar{x}\|_2^2
  &= \E \Big\| \frac{C_\a'}{n_2} \sum_{j=1}^{n_2/C_\a'} (Y_j - \E Y_j) \Big\|_2^2 
   = \Big( \frac{C_\a'}{n_2} \Big)^2 \sum_{j=1}^{n_2/C_\a'} \E \|Y_j - \E Y_j\|_2^2 \\
  &\lesssim_\a \frac{1}{n_2} \E \|Y_1 - \E Y_1\|_2^2 
   = \frac{1}{n_2} \E \|V - \E V\|_2^2 
   \le \frac{4}{n_2} \E \|V\|_2^2 \\ 
  &= \frac{4}{n_2} \sum_{i \in I_1'} \lambda_i \|X_i\|_2^2/a_i^2
   \le \frac{4}{n_2} \max_{i \in I_1} \|X_i\|_2^2/a_i^2,
\end{align*}
where the last inequality follows because $I_1' \subseteq I_1$ 
and $\sum_{i \in I_1} \lambda_i \le 1$.

Since $\bar{n} \ge n$, \eqref{ai on I1} gives us the lower bound 
$$
a_i^2 \ge \frac{n}{l n_1}		\quad \text{for } i \in I_1.
$$
Together with the assumption $\|X_i\|_2^2 \le n$, this implies that 
$$
\E \|\bar{y} - \bar{x}\|_2^2 
  \lesssim_\a \frac{1}{n_2} \cdot \frac{n}{n/l n_1}
  \le \frac{l n_1}{n_2}.
$$
Since $\|\bar{x}\|_2^2 = 1 \le ln_1/n_2$, we conclude that with 
probability at least $0.9$, one has
\begin{equation}																\label{y norm}
  \|\bar{y}\|_2^2 \lesssim_\a \frac{l n_1}{n_2}.
\end{equation}

\smallskip
{\bf Step 4: removal of the diagonal term.}
We know from \eqref{xi ai x} that $\< X_i/a_i, \bar{x} \> \ge 1$ for many terms $X_i$.
We would like to replace $\bar{x}$ by its random version $\bar{y}$, establishing a lower bound
$\< X_k/a_k, \bar{y} \> \ge 1$ for many terms $X_k$. 
But at the same time, our main goal is decoupling, 
in which we would need to make the random vector $\bar{y}$ independent
of those terms $X_k$.
To make this possible, we will first remove from the sum \eqref{y bar} defining $\bar{y}$
the ``diagonal'' term containing $X_k$, and we call the resulting vector $\bar{y}^{(k)}$.

To make this precise, let us fix $k \in I_2' \subseteq I_1' \subseteq I_1$. 
We consider independent random vectors 
$$
V_j^{(k)} := V_j \one_{\{V_j \ne X_k/a_k\}}, 
\quad
Y_j^{(k)} := V_j^{(k)} + \sum_{i \in I_1 \setminus I_1'} \lambda_i X_i a_i, 
\quad j=1,2,\ldots
$$
Note that 
\begin{equation}										\label{discrepancy}
  \P \{ Y_j^{(k)} \ne Y_j \} = \P \{ V_j^{(k)} \ne V_j \} 
  = \P \{ V_j = X_k/a_k \} = \lambda_k.
\end{equation}
Similarly to the definition \eqref{y bar} of $\bar{y}$, we define
$$
\bar{y}^{(k)} := \frac{C_\a'}{n_2} \sum_{j=1}^{n_2/C_\a'} Y_j^{(k)}.
$$
Then 
\begin{equation}									\label{ykx}
  \E \bar{y}^{(k)} 
  = \E Y_1^{(k)}
  = \E Y_1 - \lambda_k X_k/a_k
  = \bar{x} - \lambda_k X_k/a_k.
\end{equation}

As we said before, we would like to show that the random variable
$$
Z_k := \< X_k/a_k, \bar{y}^{(k)} \>
$$
is bounded below by a constant with high probability. 
First, we will estimate its mean
$$
\E Z_k = \< X_k/a_k, \bar{x} \> - \lambda_k \|X_k\|_2^2/a_k^2.
$$
To estimate the terms in the right hand side, note that 
$\< X_i/a_i, \bar{x} \> \ge 1$ by \eqref{xi ai x}
and $\|X_k\|_2^2 \le n$ by the assumption.
Now is the crucial point when we use that $a_i^2$ dominate $\lambda_i$
as in the second inequality in \eqref{ai on I2}. This allows us to bound the 
``diagonal'' term as
$$
\lambda_k \|X_k\|_2^2/a_k^2 \le n/2\bar{n} \le n/2n = 1/2.
$$
As a result, we have
\begin{equation}										\label{exp}
  \E Z_k \ge 1 - 1/2 = 1/2.
\end{equation}

\smallskip
{\bf Step 5: control of the inner products.}
We would need a stronger statement than \eqref{exp} --
that $Z_k$ is bounded below not only in expectation
but also with high probability. We will get this immediately by 
Chebyshev's inequality if we can upper bound the variance of $Z_k$. 
In a way similar to Step~3, we estimate
\begin{align}											\label{var}
\Var Z_k 
  &= \E (Z_k - \E Z_k)^2
  = \E \Big\langle X_k/a_k, \frac{C_\a'}{n_2} 
    \sum_{j=1}^{n_2/C_\a'} (Y_j^{(k)} - \E Y_j^{(k)}) \Big\rangle ^2 \nonumber\\
  &\lesssim_\a \frac{1}{n_2} \E \< X_k/a_k, V_1^{(k)} \> ^2
  = \frac{1}{n_2} \sum_{i \in I_1',\; i \ne k} \lambda_i \< X_k/a_k, X_i/a_i \> ^2.
\end{align}
Now we need to estimate the various terms in the right hand side of \eqref{var}. 

We start with the estimate on the inner products, collecting them into
$$
S := \sum_{i \in I_1',\; i \ne k} \lambda_i \< X_k, X_i \> ^2
= \sum_{i \in I_1'} \lambda_i p_{ki}
\quad \text{where } p_{ki} = \< X_k, X_i \> ^2 \one_{\{k \ne i\}}.
$$
Recall that, by the construction of $\lambda_i$ and of $I_1' \subseteq I$, we
have $\sum_{i \in I_1'} \lambda_i \le 1$ and $\lambda_i \le C_\a/n_2$ for $i \in I_1'$.
We use Lemma~\ref{order statistics} on order statistics to obtain the bound
$$
S \le \frac{C_\a}{n_2} \sum_{i=1}^{n_2/C_\a} (p_k)^*_i 
  = \frac{C_\a}{n_2} \max_{\substack{E \subseteq [m] \\ |E| = n_2/C_\a}} 
      \sum_{i \in E,\; i \ne k} \< X_i, X_k \> ^2.
$$
Finally, we use our weak orthonormality assumption \eqref{weak orthonormality}
to conclude that
$$
S  \lesssim_\a \Big( \frac{N}{n_2} \Big)^{1/r'} n.
$$

To complete the bound on the variance of $Z_k$ in \eqref{var} it 
remains to obtain some good lower bounds on $a_k$ and $a_i$.
Since $k \in I_2' \subseteq I_2$, \eqref{ai on I2} yields
$$
a_k^2 \ge \frac{\bar{n}}{l n_2} \ge \frac{n}{l n_2}.
$$
Similarly we can bound the coefficients $a_i$ in \eqref{var}:
using \eqref{ai on I1} we have $a_i^2 \ge \bar{n}/{l n_1}$ since since $i \in I_1' \subseteq I$.
But here we will not simply replace $\bar{n}$ by $n$, as we shall try to use
$a_i^2$ to offset the term $(N/n_2)^{1/r'}$ in the estimate on $S$.
To this end, we note that $\bar{n} \ge (N/m)^{1/r} m \ge (N/n_1)^{1/r} n_1$
because $m \ge n_1$. Therefore, using the last inequality in \eqref{n1 n2} and that $N \ge m$, 
we have
\begin{equation}										\label{nbar n1}
  \frac{\bar{n}}{n_1} \ge \Big( \frac{N}{n_1} \Big)^{1/r} 
  \ge \Big( \frac{N}{n_2} \Big)^{\frac{1}{(1+\a)r}}.
\end{equation}
Using this, we obtain a good lower bound 
$$  
a_i^2 \ge \frac{\bar{n}}{l n_1}
\ge \frac{1}{l} \Big( \frac{N}{n_2} \Big)^{\frac{1}{(1+\a)r}},
\quad \text{for } i \in I_1'.
$$

Combining the estimates on $S$, $a_k$ and $a_i$, we conclude our lower bound \eqref{var}
on the variance of $Z_k$ as follows:
\begin{align*}
\Var Z_k
  &\lesssim_\a \frac{1}{n_2} 
    \cdot \frac{ln_2}{n} 
    \cdot l \Big( \frac{n_2}{N} \Big)^{\frac{1}{(1+\a)r}}
    \cdot \Big( \frac{N}{n_2} \Big)^{1/r'} n \\
  &\le l^2 \Big( \frac{n_2}{N} \Big)^\a		
  			\quad \text{(by choosing $\a$ small enough depending on $r$, $r'$)}\\
  &\le l^2 \Big( \frac{n_2}{m} \Big)^\a
  \le l^2 2^{-\a l / 2}					\quad \text{(by \eqref{n1 n2})}\\
  &\le \a/16.			\quad \text{(since $l$ is large enough depending on $\a$)}
\end{align*}
Combining this with the lower bound \eqref{exp} on the expectation, we conclude 
by Chebyshev's inequality the desired estimate
\begin{equation}										\label{Zk large}
  \P \big\{ Z_k = \< X_k/a_k, \bar{y}^{(k)} \> \ge \frac{1}{4} \big\}
  \ge 1-\a
  \quad \text{for } k \in I_2'.
\end{equation}

\smallskip
{\bf Step 6: decoupling.}
We are nearing the completion of the proof.
Let us consider the good events
$$
\EE_k := \big\{ \< X_k/a_k, \bar{y} \> \ge \frac{1}{4}
  \text{ and } \bar{y} = \bar{y}^{(k)} \big\}
\quad \text{for } k \in I_2'.
$$
To show that each $\EE_k$ occurs with high probability, 
we note that by definition of $\bar{y}$ and $\bar{y}^{(k)}$ one has
\begin{align*}
\P \{ \bar{y} \ne \bar{y}^{(k)} \}
  &\le \P \{ Y_j \ne Y_j^{(k)} \text{ for some } j \in [n_2/C_\a']\} \\
  &\le \sum_{j=1}^{n_2/C_\a'} \P \{ Y_j \ne Y_j^{(k)} \}
  = \frac{n_2}{C_\a'} \cdot \lambda_k				\quad \text{(by \eqref{discrepancy})} \\
  &\le \frac{n_2}{C_\a'} \cdot \frac{C_\a}{n_2}		\quad \text{(by definition of $I_2'$)} \\  
  &= \a.											\quad \text{(as we chose $C_\a' = C_\a/\a$)}
\end{align*}
From this and using \eqref{Zk large} we conclude that
$$
\P \{ \EE_k \} 
\ge 1 - \P \big\{ \< X_k/a_k, \bar{y}^{(k)} \> < \frac{1}{4} \big\} 
  - \P \{ \bar{y} \ne \bar{y}^{(k)} \}
\ge 1 - 2\a
\quad \text{for } k \in I_2'.
$$

An application of Fubini theorem yields that with probability at least $0.9$,
at least $(1-20\a)|I_2'|$ of the events $\EE_k$ hold simultaneously. 
More accurately, with probability at least $0.9$ the following event occurs, 
which we denote by $\EE$. 
There exists a subset $I \subseteq I_2'$ of size $|I| \ge (1-20\a)|I_2'|$ such that 
$\EE_k$ holds for all $k \in I$. Note that using \eqref{I minus I} and choosing $C_\a$
sufficiently large we have
\begin{equation}										\label{I size}
  (1-21\a)n_2 \le |I| \le n_2.
\end{equation}

Recall that the norm bound \eqref{y norm} also holds 
with high probability $0.9$. Hence with probability at least $0.8$, both 
$\EE$ and this norm bound holds.
Let us fix a realization of our random variables for which this happens.
Then, first of all, by definition of $\EE_k$ we have
\begin{equation}										\label{xk ak y}
  \< X_k/a_k, \bar{y} \> \ge \frac{1}{4}
  \quad \text{for } k \in I.
\end{equation}
Next, we are going to observe that $\bar{y}$ lies in the span of few vectors $X_i$.
Indeed, by construction $\bar{y}^{(k)}$ lies in the span of the vectors 
$Y_j^{(k)}$ for $j \in [n_2/C_\a']$.
Each such $Y_j^{(k)}$ by construction lies in the span of the vectors 
$X_i$, $i \in I_1 \setminus I_1'$
and of one vector $V_j^{(k)}$. Finally, each such vector $V_j^{(k)}$, again by construction, 
is either equal zero or $V_j$, which in turn equals $X_{i_0}$ for some $i_0 \ne k$.
Since $\EE$ holds, we have $\bar{y} = \bar{y}^{(k)}$ for all $k \in I$. This implies that there 
exists a subset $I_0 \subseteq [m]$ (consisting of the indices $i_0$ as above)
with the following properties. Firstly, $I_0$ does not contain any of indices $k \in I$; 
in other words $I_0$ is disjoint from $I$. Secondly, this set is small: $|I_0| \le n_2/C_\a'$.
Thirdly, $\bar{y}$ lies in the span of $X_i$, $i \in I_0 \cup (I_1 \setminus I_1')$. 
We claim that this set of indices, 
$$
J := I_0 \cup (I_1 \setminus I_1')
$$
satisfies the conclusion of the Proposition.

Since $I$ and $I_0$ are disjoint and $I \subseteq I_2' \subseteq I_1'$, 
it follows that $I$ and $J$ are disjoint as required.
Moreover, by \eqref{I minus I} and by choosing $C_\a$, $C_\a'$ sufficiently large we have
$$
|J| \le |I_0| + |I_1 \setminus I_1'|
\le \frac{n_2}{C_\a'} + \frac{n_2}{C_\a}
\le \a n_2.
$$
When we combine this with \eqref{I size} and choose $\a$ sufficiently small depending on $\d$,
we achieve
$$
|J| \le \d |I|
$$
as required.
Finally, we claim that the normalized vector 
$$
y := \frac{\bar{y}}{\|\bar{y}\|_2}
$$
satisfies the conclusion of the Proposition.
Indeed, we already noted that $\bar{y} \in \Span(X_j)_{j \in J}$, as required. 
Next, for each $k \in I \subseteq I_2' \subseteq I_2$ we have
\begin{align*}
\< X_k, y \> ^2 
  &\ge \frac{a_k^2}{16\|\bar{y}\|_2^2}					\quad \text{(by \eqref{xk ak y})} \\
  &\gtrsim_\a \frac{\bar{n}}{l n_2} \cdot \frac{n_2}{l n_1}	
    				\quad \text{(by \eqref{ai on I2} and \eqref{y norm})} \\
  &= \frac{\bar{n}}{l^2 n_1}
  \ge \frac{1}{l^2} \Big( \frac{N}{n_2} \Big)^{\frac{1}{(1+\a)r}}.
  				\quad \text{(by \eqref{nbar n1})}
\end{align*}
We can get rid of $l^2$ in this estimate using the bound
\begin{align*}
\Big( \frac{N}{n_2} \Big)^{\frac{\a}{(1+\a)r}}
  &\ge \Big( \frac{m}{n_2} \Big)^{\frac{\a}{(1+\a)r}} 
  \ge 2^{\frac{\a l}{2(1+\a)r}}				\quad \text{(by \eqref{n1 n2})} \\
  &\ge 2^{\frac{\a l}{3 r}} 
  \ge 2^{\a^2 l}					\quad \text{(choosing $\a$ small enough depending on $r$)}\\
  &\ge l^2. 						\quad \text{(since $l$ is large enough depending on $\a$)}
\end{align*}
Therefore
$$
\< X_k, y \> ^2 
\ge \Big( \frac{N}{n_2} \Big)^{\frac{1-\a}{(1+\a)r}}
\ge \Big( \frac{N}{n_2} \Big)^{1/r''} 
\quad \text{for } k \in I
$$
where the last inequality follows by choosing $\a$ sufficiently small depending on $r$, $r''$.
This completes the proof of Decoupling Proposition~\ref{decoupling}.
\end{proof}

\section{Norms of random matrices with independent columns}		\label{s: random matrices}

In this section we apply our decoupling principle, Proposition~\ref{decoupling},
to estimate norms of random matrices with independent columns.
As we said, a simple truncation argument of J.~Bourgain \cite{Bo} reduces the
approximation problem for covariance matrices to bounding the norm of the 
random matrix $\sum_{i \in E} X_i \otimes X_i$
uniformly over index sets $E$. The following result gives such an estimate
for random vectors $X_i$ with finite moment assumptions. 

\begin{theorem}							\label{norm}
  Let $1 \le n \le N$ be integers, and let $4 < p < q$ and $t \ge 1$.
  Consider independent random vectors $X_1,\ldots,X_N$ in $\R^n$ ($1 \le n \le N$) which satisfy 
  the moment assumptions \eqref{moment assumptions}.
  Then with probability at least $1 - C t^{-0.9q}$,
  for every index set $E \subseteq [N]$, $|E| \ge 4$, one has
  $$
  \Big\| \sum_{i \in E} X_i \otimes X_i \Big\|
    \lesssim_{p,q,K,L} t^2 (\log \log |E|)^2 \Big[ n + \Big( \frac{N}{|E|} \Big)^{4/p} |E| \Big].
  $$
\end{theorem}

We can state Theorem~\ref{norm} in terms of random matrices with independent columns.

\begin{corollary}							\label{independent columns}
  Let $1 \le n \le N$ be integers, and let $4 < p < q$ and $t \ge 1$.
  Consider the $n \times N$ random matrix $A$ whose columns are
  independent random vectors $X_1,\ldots,X_N$ in $\R^n$ which satisfy \eqref{moment assumptions}.
  Then with probability at least $1 - C t^{-0.9q}$ one has 
  $$
  \|A\| \lesssim_{p,q,K,L} t \log \log N \cdot (\sqrt{n} + \sqrt{N}).
  $$
  Moreover, with the same probability all $n \times m$ submatrices $B$ of $A$ simultaneously 
  satisfy the following for all $4 \le m \le N$:
  $$
  \|B\| \lesssim_{p,q,K,L} t \log \log m \cdot  \Big[ \sqrt{n} + \Big( \frac{N}{m} \Big)^{2/p} \sqrt{m} \Big].
  $$
\end{corollary}

\begin{proof}[Proof of Theorem~\ref{norm}.]
By replacing $X_i$ with $X_i / \max(K,L)$ we can assume without loss of generality that $K=L=1$.
As we said, the argument will be based on Decoupling Proposition~\ref{decoupling}.
Its assumptions follow from known results. 
Indeed, the pairwise almost orthogonality of the vectors $X_i$ follows 
from Lemma~\ref{almost orthogonality}, which yields \eqref{cross terms OK}
with probability at least $1 - C t^{-q}$.
Also, the required bound on the weak $\ell_2$ norm follows from Theorem~\ref{weak ell2 bound},
which gives with probability at least $1 - C t^{-0.9 q}$ that
\begin{equation}							\label{2 infty OK}
  \sup_{x \in S^{n-1}} \|( \< X_i, x\> )_{i \in I}\|_{2,\infty}^2 
    \lesssim_{q} t^2 \Big[ n + \Big( \frac{N}{|I|} \Big)^{4/q} |I| \Big]
    \quad \text{for } I \subseteq [N].
\end{equation}
Consider the event $\EE$ that both required bounds \eqref{cross terms OK} and \eqref{2 infty OK} hold.

Let $\EE_0$ denote the event in the conclusion of the Theorem. It remains to prove that 
$\P(\EE_0^c \text{ and } \EE)$ is small. To this end, assume that $\EE$ holds but $\EE_0$ does not. 
Then there exists an index set $E \subset [N]$ whose size we denote by $m := |E|$, and which satisfies
$$
\Big\| \sum_{i \in E} X_i \otimes X_i \Big\|
  = \sup_{x \in S^{n-1}} \sum_{i \in E} \< X_i, x\> ^2
  \gtrsim_{p,q} t^2 (\log \log m)^2 \Big[ n + \Big( \frac{N}{m} \Big)^{4/p} m \Big].
$$
Recalling \eqref{cross terms OK} and \eqref{2 infty OK} we see that the assumptions of 
Decoupling Proposition~\ref{decoupling} hold for $1/r = 4/p$, $1/r' = 4/q$, $r'' = r'$, 
$K_1 = K = 1$, $K_2 = C_q \sqrt{t}$ for suitably large $C_q$, 
$K_3 = \max(K_1, K_2, 100 t^2)$, 
and for $\d =  \d(p,q) > 0$ sufficiently small (to be chosen later).
Applying Decoupling Proposition~\ref{decoupling} we obtain 
disjoint index sets $I, J \subseteq E \subseteq [N]$ with sizes
$$
|I| =: s, \quad  |J| \le \d s,
$$ 
and a vector $y \in S^{n-1} \cap \Span(X_j)_{j \in J}$ such that 
\begin{equation}							\label{Xi on y large}
\< X_i, y\> ^2 \ge 100 t^2 \Big( \frac{N}{s} \Big)^{1/r''}
\quad \text{for } i \in I.
\end{equation}

We will need to discretize the set of possible vectors $y$. Let 
$$
\e := \Big( \frac{\d s}{N} \Big)^5
$$
and consider an $\e$-net $\NN_J$ of the sphere $S^{n-1} \cap \Span(X_j)_{j \in J}$.
As in known by a volumetric argument (see e.g. \cite{MS} Lemma~2.6), one can choose such a net
with cardinality
$$
|\NN_J| \le (3/\e)^{|J|} \le \Big( \frac{2N}{\d s} \Big)^{5 \d s}.
$$
We can assume that the random set $\NN_J$ depends only on the number $\e$, the set $J$ and 
the random variables $(X_j)_{j \in J}$. 
Given a vector $y$ as we have found above, we can approximate it with some vector $y_0 \in \NN_J$ 
so that $\|y-y_0\|_2 \le \e$. By \eqref{2 infty OK} we have
$$
\|( \< X_i, y-y_0\> )_{i \in I}\|_{2,\infty}^2 
  \lesssim_{q} \e^2 t^2 \Big[ n + \Big( \frac{N}{s} \Big)^{1/r'} s \Big].
$$
This implies that all but at most $\d s$ indices $i$ in $I$ satisfy the inequality
\begin{equation}							\label{fraction of indices}
\< X_i, y-y_0\> ^2
  \lesssim_{q} \frac{\e^2 t^2}{\d s} \Big[ n + \Big( \frac{N}{s} \Big)^{1/r'} s \Big].
\end{equation}
Let us denote the set of these indices by $I_0 \subseteq I$.
The bound in \eqref{fraction of indices} can be simplified as
$$
\frac{\e^2}{\d s} \Big[ n + \Big( \frac{N}{s} \Big)^{1/r'} s \Big] 
  \le 2\d.
$$
Indeed, this estimate follows from the two bounds
\begin{gather*}
\frac{\e^2}{\d s} \cdot n 
  \le \Big( \frac{\d s}{N} \Big)^{10} \cdot \frac{n}{\d s} 
  \le \d		\quad \text{(because $n \le N$)}; \\
\frac{\e^2}{\d s} \cdot \Big( \frac{N}{s} \Big)^{1/r'} s 
  \le \frac{1}{\d} \Big( \frac{\d s}{N} \Big)^{10} \Big( \frac{N}{s} \Big)^{1/r'}
  \le \d.		\quad \text{(because $\d \le 1$, $r' \ge 1$)}
\end{gather*}
In particular, by choosing $\d = \d(q) > 0$ sufficiently small, \eqref{fraction of indices} implies  
$$
| \< X_i, y-y_0\> | \le t 	\quad \text{for } i \in I_0.
$$
Together with \eqref{Xi on y large} this yields by triangle inequality that 
$$
| \< X_i, y_0\> |  
  \ge 10 t \Big( \frac{N}{s} \Big)^{1/2r''} - t 
  \ge 9 t \Big( \frac{N}{s} \Big)^{1/2r''}		\quad \text{for } i \in I_0.
$$

Summarizing, we have shown that the event $\{ \EE_0^c \text{ and } \EE \}$ implies 
the following event: there exists a number $s \le N$, disjoint index subsets $I_0, J \subseteq [N]$
with sizes $|I_0| \ge (1-\d)s$, $|J| \le \d s$, and a vector $y_0 \in N_J$ such that 
$$
| \< X_i, y_0\> |  \ge 9 t \Big( \frac{N}{s} \Big)^{1/2r''}		\quad \text{for } i \in I_0.
$$
It will now be easy to estimate the probability of this event. 
First of all, for each fixed vector $y_0 \in S^{n-1}$ and each index $i$, 
the moment assumptions \eqref{moment assumptions} 
imply via Markov's inequality that 
$$
\P \big\{ | \< X_i, y_0\> |  \ge 9 t \Big( \frac{N}{s} \Big)^{1/2r''} \big\}
  \le \frac{1}{(9t)^q} \Big( \frac{N}{s} \Big)^{-q/2r''}
  \le \frac{1}{9 t^q} \Big( \frac{N}{s} \Big)^{-2}
$$
where the last line follows from our choice of $q$ and $r''$.
By independence, for each fixed vector $y_0 \in S^{n-1}$ and a fixed index set $I_0 \subseteq[N]$
of size $|I_0| \ge (1-\d)s$ we have
\begin{align}		
\P \big\{ | \< X_i, y_0\> |  \ge 9 \Big( \frac{N}{s} \Big)^{1/2r''} \text{ for } i \in I_0 \big\}
  &\le \Big[ \frac{1}{9 t^q} \Big( \frac{N}{s} \Big)^{-2} \Big]^{|I_0|} 	\nonumber\\
  \le 9^{-(1-\d)s} t^{-(1-\d)q} \Big( \frac{N}{s} \Big)^{-2(1-\d) s}.					\label{prob for fixed} 
\end{align}
Then we bound the probability of event $\{ \EE_0^c \text{ and } \EE \}$ by taking the union bound
over all $s$, $I_0$, $J$ as above, conditioning on the random variables $(X_j)_{j \in J}$ 
(which fixes the $\e$-net $\NN_J$), taking the union bound over the choice of $y_0 \in \NN_J$, 
and finally evaluating the probability for using \eqref{prob for fixed}. This way we obtain via
Stirling's approximation of the binomial coefficients that 
\begin{align*}
\P \{ \EE_0^c \text{ and } \EE \}
  &\le \sum_{s=1}^N \binom{N}{|I_0|} \binom{N}{|J|} |\NN_J| 
    \; 9^{-(1-\d)s} t^{-(1-\d)q} \Big( \frac{N}{s} \Big)^{-2(1-\d) s} \\
  &\le t^{-(1-\d)q} \sum_{s=1}^N 
    \Big( \frac{eN}{(1-\d)s} \Big)^{(1-\d)s} 
    \Big( \frac{eN}{\d s} \Big)^{\d s}
    \Big( \frac{2N}{\d s} \Big)^{5 \d s}
    \; 9^{-(1-\d)s} \Big( \frac{N}{s} \Big)^{-2(1-\d) s} \\
  &\le t^{-0.9 q} \sum_{s=1}^N \Big( \frac{N}{2 s} \Big)^s		\quad \text{(by choosing $\d>0$ small enough)}\\
  &\le t^{-0.9 q}.
\end{align*}
It follows that 
$$
\P \{ \EE_0^c \}
\le \P \{ \EE_0^c \text{ and } \EE \} + \P \{ \EE^c \}
\le t^{-0.9 q} + C t^{-q} + C t^{-0.9q}
\lesssim t^{-0.9 q}.
$$
This completes the proof of Theorem~\ref{norm}.
\end{proof}

\section{Approximating covariance matrices}					\label{s: approximation}

In this final section, we deduce our main result on the approximation of covariance
matrices for random vectors with finite moments.

\begin{theorem}													\label{main full}
  Consider independent random vectors $X_1,\ldots,X_N$ in $\R^n$, $4 \le n \le N$, which satisfy 
  moment assumptions \eqref{moment assumptions} for some $q > 4$ and some $K$, $L$.
  Then for every $\d > 0$ with probability at least $1-\d$ one has
  \begin{equation}																	\label{eq deviation}
  \Big\| \frac{1}{N} \sum_{i=1}^N X_i \otimes X_i - \E X_i \otimes X_i \Big\| 
    \lesssim_{q,K,L,\d} (\log \log n)^2 \Big( \frac{n}{N} \Big)^{\frac{1}{2}-\frac{2}{q}}.
  \end{equation}  
\end{theorem}

In our proof of Theorem~\ref{main full}, we can clearly assume that $K = L = 1$ in 
the moment assumptions \eqref{moment assumptions} by rescaling the vectors $X_i$.
So in the rest of this section we suppose $X_i$ are such random vectors. 

For a level $B>0$ and a vector $x \in S^{n-1}$, we consider the (random) 
index set of large coefficients
$$
E_B = E_B(x) := \{i \in [N] :\; |\< X_i, x\> | \ge B\}.
$$

\begin{lemma}[Large coefficients]				\label{large coefficients}
  Let $t \ge 1$. With probability at least $1 - C t^{0.9q}$, one has
  $$
  |E_B| \lesssim_q n/B^2 + N(t/B)^{q/2}		\quad \text{for } B > 0.
  $$
\end{lemma}

\begin{proof}
This estimate follows from Theorem~\ref{weak ell2 bound}.
By definition of the set $E_B$ and the weak $\ell_2$ norm, we obtain 
with the required probability that 
$$
B^2 |E_B| 
  \le \|(\< X_i, x\> )_{i \in E_B}\|_{2,\infty}^2
  \lesssim_q n + t^2 \Big( \frac{N}{|E_B|} \Big)^{4/q} |E_B|.
$$
Solving for $|E_B|$ we obtain the bound as in the conclusion.
\end{proof}

\medskip
\begin{proof}[Proof of Theorem~\ref{main full}.]
The truncation argument described in \cite{ALPT} in the beginning of proof of Proposition~4.3
reduces the problem to estimating the contribution to the sum of large coefficients.
Denote
$$
E =  \Big\| \frac{1}{N} \sum_{i=1}^N X_i \otimes X_i - \E X_i \otimes X_i \Big\|
= \sup_{x \in S^{n-1}} \Big| \frac{1}{N} \sum_{i=1}^N \< X_i, x\> ^2 - \E \< X_i, x\> ^2 \Big|.
$$
The truncation argument yields
that for every $B \ge 1$, one has with probability at least $1 - \d/3$ that
\begin{align}								\label{three terms}
E &\lesssim_{q,\d} B \sqrt{\frac{n}{N}} 
  + \sup_{x \in S^{n-1}} \frac{1}{N} \sum_{i \in E_B} \< X_i, x\> ^2
  + \sup_{x \in S^{n-1}} \frac{1}{N} \E \sum_{i \in E_B} \< X_i, x\> ^2 \nonumber\\
  &=: I_1 + I_2 + I_3.
\end{align}
We choose the value of the level 
$$
B = \Big( \frac{N}{n} \Big)^{2/q}
$$
so that, using Lemma~\ref{large coefficients}, with probability at least $1-\d/3$
we have 
\begin{equation}							\label{EB size}
|E_B| \lesssim_{q,\d} n.
\end{equation}
It remains to estimate the right hand side of \eqref{three terms} using \eqref{EB size}.

First, we clearly have
$$
I_1 = \Big( \frac{n}{N} \Big)^{\frac{1}{2} - \frac{2}{q}}.
$$
An estimate of $I_2$ follows from Theorem~\ref{norm} for some $p = p(q) \in (4,q)$ 
to be determined later. Note that enlarging $E_B$ can only make $I_2$ and $I_3$ larger. 
So without loss of generality we can assume that $|E_B| \ge 4$ as required in Theorem~\ref{norm}.
This way, we obtain with probability at least $1 - \d/3$ that 
\begin{align*}
I_2 
  &\lesssim_{q,\d} \frac{1}{N} (\log \log |E_B|)^2 
    \Big[ n + \Big( \frac{N}{|E_B|} \Big)^{4/p} |E_B| \Big] \\
  &\lesssim_{q,\d} (\log \log n)^2 \Big[ \frac{n}{N} + \Big( \frac{n}{N} \Big)^{1-\frac{4}{p}} \Big].
  		\qquad \text{(by \eqref{EB size})}
\end{align*}
Finally, to estimate $I_3$ let us fix $x$ and consider the random variable $Z_i = |\< X_i, x\> |$. 
Since $\E Z_i^q \le 1$, an application of H\"older's and Markov's inequalities yield 
$$
\E Z_i^2 \one_{\{Z_i \ge B\}} 
  \le (\E Z_i^q)^{2/q} ( \P(Z_i \ge B) )^{1-2/q} 
  \le B^{2-q}
  \lesssim_{q,\d} \Big( \frac{n}{N} \Big)^{2-\frac{4}{q}}.
$$
Therefore
$$
I_3
  = \sup_{x \in S^{n-1}} \frac{1}{N} \sum_{i=1}^N \E Z_i^2 \one_{\{Z_i \ge B\}}
  \lesssim_{q,\d} \Big( \frac{n}{N} \Big)^{2-\frac{4}{q}}.
$$

Since we are free to choose $p = p(q)$ in the interval $(4,q)$, we choose the middle of the interval, 
$p = (q+4)/2$. Returning to \eqref{three terms} we conclude that 
$$
E \lesssim_{q,\d} (\log \log n)^2 \Big( \frac{n}{N} \Big)^{\frac{1}{2} - \frac{2}{q}}.
$$
This completes the proof of Theorem~\ref{main full}.
\end{proof}

\end{document}